\documentclass[14pt,reqno]{amsart}
\usepackage{stix}
\usepackage{amsmath,amsxtra,color,calligra,mathrsfs}
\usepackage{xcolor}
\colorlet{mdtRed}{red!50!black}
\definecolor{dblue}{rgb}{0,0,.6}
\usepackage[colorlinks]{hyperref}
\hypersetup{linkcolor=mdtRed,citecolor=dblue,filecolor=dullmagenta,urlcolor=mdtRed}
\usepackage[all]{xy}
\usepackage{tikz,tikz-cd,tkz-graph,enumerate}
\usetikzlibrary{matrix,arrows,decorations.pathmorphing}

\DeclareMathOperator{\Pic}{\textnormal{Pic}}
\DeclareMathOperator{\Br}{\textnormal{Br}}

\DeclareMathOperator{\Mxi}{M_{\xi}}
\DeclareMathOperator{\MLxi}{M^L_{\xi}}
\DeclareMathOperator{\Malpha}{M_{\alpha}}
\DeclareMathOperator{\Mbeta}{M_\beta}
\DeclareMathOperator{\PNalpha}{\textnormal{PN}_{\alpha,i}}

\DeclareMathOperator{\MLalpha}{M^L_{\alpha}}
\DeclareMathOperator{\codim}{\textnormal{codim}}
\DeclareMathOperator{\BY}{H^2 (Y_{\text{\'et}},\, \mathbb G_m)_{\text{torsion}}}
\DeclareMathOperator{\BYY}{H^2 (Y_{\text{\'et}},\, \mathbb G_m)}

\newtheorem{theorem}{Theorem}[section]
\newtheorem{lemma}[theorem]{Lemma} 
\newtheorem{proposition}[theorem]{Proposition}
\newtheorem{corollary}[theorem]{Corollary}
\newtheorem{claim}{Claim}[section]

\theoremstyle{definition}
\newtheorem{definition}[theorem]{Definition}

\numberwithin{equation}{section} 

\begin{document}
	
	\baselineskip=15.5pt 
	
	\title[Brauer group of moduli stack of stable parabolic $\textnormal{PGL}(r)$-bundles]{Brauer group of moduli stack
		of stable parabolic $\textnormal{PGL}(r)$-bundles over a curve}
	
	\author[I. Biswas]{Indranil Biswas}
	
\address{Mathematics Department, Shiv Nadar University, NH91, Tehsil Dadri,
Greater Noida, Uttar Pradesh 201314, India}

\email{indranil.biswas@snu.edu.in, indranil29@gmail.com}

\author[S. Chakraborty]{Sujoy Chakraborty}

	\address{Department of Mathematics, Indian Institute of Technology Chennai, Chennai, India}
	\email{sujoy.cmi@gmail.com}

\author[A. Dey]{Arijit Dey}

	\address{Department of Mathematics, Indian Institute of Technology Chennai, Chennai, India}
	\email{arijitdey@gmail.com}


	
	\begin{abstract}
		Let $k$ be an algebraically closed field of characteristic 
		zero. We prove that the Brauer group of the moduli stack of stable parabolic
$\textnormal{PGL}(r,k)$-bundles on a curve $X$, with full flag quasi-parabolic structures at an arbitrary 
		parabolic divisor, coincides with the Brauer group of the smooth locus of the corresponding coarse 
		moduli space of parabolic $\textnormal{PGL}(r,k)$-bundles.
		We also compute the Brauer group of the smooth locus of this coarse moduli for more general quasi-parabolic types and weights satisfying certain conditions.
	\end{abstract}
	
	\maketitle
	
	\section{introduction}
	
	Let $Y$ be a quasi-projective variety defined over an algebraically closed field $k$ of 
	characteristic zero. The cohomological Brauer group of $Y$ is defined to be the torsion part
	$$\BY\, \subset\, H^2 (Y_{\text{\'et}},\, \mathbb G_m),$$ and it is denoted by $\Br'(Y)$. When
	the variety $Y$ is smooth, it is known that the above group $\BYY$ is
	actually torsion. The Brauer group of $Y$, 
	which is denoted by $\Br(Y)$, is defined to be the Morita equivalence classes of Azumaya algebras over $Y$.  Giving an Azumaya algebra over $Y$ 
	is equivalent to giving a Brauer-Severi scheme (i.e., a projective bundle) over $Y$ in the \'etale topology, which is also equivalent to 
	giving a principal $\textnormal{PGL}(r,k)$-bundle over $Y$ for some $r$. The Brauer class of a
	projective bundle $\mathbb{P}\, \longrightarrow\, Y$ can be thought of as the 
	obstruction for $\mathbb{P}$ to be the projectivization of a vector bundle on $Y$. For
	each integer $r\, \geq\, 2$, the sequence of cohomologies associated to the exact sequence
	\[
	\{e\} \,\longrightarrow\, \mathbb \mu_{r}\,\longrightarrow\, \text{SL}_r\,\longrightarrow\, \text{PGL}_r\,\longrightarrow\, \{e\}
	\]
	gives a homomorphism $H^1(Y_{\text{\'et}}, \,\textnormal{PGL}_r)\,\longrightarrow\,
	\text{Br}'(Y)$, which induces an injective group
	homomorphism $$i\, :\, \text{Br}(Y)\,\longrightarrow\, \Br'(Y).$$ By a theorem of
	Gabber it is known that this homomorphism $i$ is surjective
	\cite{dJ13}. For an algebraic stack $Y$, by Brauer group of $Y$ we mean the cohomological Brauer group of $Y$.
	
	Let $X$ be an irreducible smooth projective curve over $k$ of genus $g$ with $g\,\geq \,2$. Fix an 
	integer $r\,\geq \,2$ and also a line bundle $\xi$ of degree $d$ on $X$. If $g\,=\,2,$ we assume that $r 
	\,\geq\, 3$. let $M(r,d)$ denote the moduli space of stable vector bundles on $X$ of rank 
	$r$ and degree $d$; it is a smooth quasi-projective variety. Let $M_{\xi}(r,d)\, \subset\, 
	M(r,d)$ denote the moduli space of stable vector bundles on $X$ with rank $r$ and 
	determinant $\xi$. Let $\mathcal M(r,d)$ and $\mathcal M_{\xi}(r,d)$ denote the 
	corresponding moduli stack of stable vector bundles. The Brauer group $\Br( M_{\xi}(r,d))$ 
	is cyclic of order g.c.d.($r,d$) \cite{BBGN07}. A generator for $\Br( M_{\xi}(r,d))$ is 
	obtained by restricting the universal projective bundle over $X \times M_{\xi}(r,d)$ to 
	$\{x_{0} \} \times M_{\xi}(r,d)$ , where $x_{0}$ is a closed point of $X$.
	
	Parabolic vector bundles on $X$, denoted by $E_*$, were introduced by Mehta and Seshadri in 
	\cite{MS80}. These are vector bundles $E$ on $X$ with a filtration of fibers of $E$ at a 
	collection of finitely many points $S$ of $X$ and certain real numbers, called weights, 
	attached to these filtrations.
	Let $M^\alpha(r,d)$ denote the moduli
	space of rank $r$ stable parabolic vector bundles $E_*$ on $X$ whose underlying bundle $E$
	is of degree $d$ and the parabolic structure 
	at each parabolic point $p_i\,\in\, S$ is of the following type: 
	$$
	E_{p_i} \,=:\, F_{i,1} \, \supsetneq\, \cdots\,
	\, \supsetneq\, F_{i,a_i}\, \not=\, 0
	$$
	with $\dim F_{i,j}/F_{i,j+1}\, :=\, r_{i,j}$; the parabolic weight of $F_{i,j}$ is
	$\alpha_{i,j}\, \in\, \mathbb R$ with
	$$
	0\, \leq\, \alpha_{i,1}\, <\, \cdots\, <\, \alpha_{i,a_i}\, <\, 1\, .
	$$
	(see \S~2 for more details). It is known that $M^\alpha(r,d)$ is a smooth
	quasi--projective variety \cite{MS80}. Let ${M}^\alpha_\xi(r,d)$ denote the moduli space
	of stable parabolic vector bundles of rank $r$ and above parabolic type with
	determinant $\xi$, i.e., $\bigwedge^rE \,=\,\xi$. Let $\mathcal{M}^\alpha(r,d)$ and
	$\mathcal {M}^\alpha_\xi(r,d)$ denote the corresponding moduli stack of rank $r$ stable
	parabolic vector bundles with above parabolic data. 
	
	The Brauer group $\Br( M^\alpha_\xi(r,d))$ is isomorphic to the cyclic group
	${\mathbb Z}/m{\mathbb Z}$ \cite{BD11}, where
	\begin{align}\label{gcd}
		m\, =\, {\rm g.c.d.}(d\,, r\, ,r_{1,1}\, , \cdots\, , 
		r_{1,a_1}\, , \cdots\, , r_{\ell,1}\, , \cdots\, ,
		r_{\ell,a_\ell}).
	\end{align}
	Like in the case of vector bundles, the cyclic group ${\rm Br}( M^\alpha_\xi(r,d))$
	is generated by the Brauer class of the
	Brauer--Severi variety ${\mathbb P}_{x_0}$, where $\mathbb P$ is the universal
	projective bundle on $X \times  M^\alpha_\xi(r,d)$ and $x_0\, \in\, X$.
	
	Let $\mathcal N(r)$ be the moduli stack of stable $\textnormal{PGL}(r,k)$ bundles and
	$N(r)$ the 
	corresponding coarse moduli space. We have the natural morphism 
	\[
	f\,:\,\mathcal N(r) \,\longrightarrow \, N(r).
	\]
	The connected components of both $\mathcal N(r)$ and $N(r)$ are indexed by integers $i \,\in\, 
	[0,\,r-1]$; the $i$-th component of
	$\mathcal N(r)$ (respectively, $N(r)$) is
	denoted by $\mathcal N(r)_{i}$ (respectively, $N(r)_{i}$). This
	$\mathcal N(r)_{i}$ (respectively, $N(r)_{i}$) is the quotient of $\mathcal 
	M_{\xi}(r,d)$ (respectively, $M_{\xi}(r,d)$) by the group of stack of $\mu_r$-bundles $\rm 
	Bun_{\mu_r}$ ((respectively, by the subgroup $\Gamma \,\subset\, \Pic^{0}(X)$
	of $r$-torsion points), where 
	$d \,\equiv\,i\,\ (\text{mod}\ r)$ (see \cite[Remark 2.8]{GO18}). The above morphism $f$
	sends $\mathcal N(r)_{i}$ onto $N(r)_{i}$. The  Brauer groups of $\mathcal N(r)_{i}$ and $N(r)_i$
	coincide \cite{BH10}, and they fit in the exact sequence
	\begin{equation}\label{ei}
		0 \,\longrightarrow\, \frac{H^2(\Gamma, \,k^{*})}{H} \,\longrightarrow\, \Br(\mathcal N(r)_{i})\,
		\longrightarrow\, \mathbb Z/\delta\mathbb Z \,\longrightarrow\, 0\, ,
	\end{equation}
	where $H\,\subset \, H^2(\Gamma,\, k^{*}) $ is a subgroup of order
	$\delta\,=\, \text{g.c.d} (r,d)$ \cite[Theorem 1.2]{BH10}.
	
	As in the case of vector bundles, a parabolic stable $\textnormal{PGL}(r,k)$-bundle is an 
	equivalence class of parabolic stable vector bundles $[E_{*}]$, where two stable parabolic 
	vector bundles $E_*$ and $E'_*$ are equivalent if there exists a line bundle $L$ such that 
	$E'_{*}\,\cong\, E_* \otimes L$. The parabolic structure on $E'_*$ is obtained from the 
	parabolic structure of $E_*$ by the above isomorphism.
	
	Let $\mathcal {PN}^{\alpha}(r)$ be the moduli stack of stable parabolic $\textnormal{PGL}(r,k)$-bundles with full-flag parabolic structures at each parabolic 
	points of $S$, and let $PN^{\alpha}(r)$ denote the corresponding coarse moduli space. We have 
	a natural morphism
	\begin{align}\label{mapping1}
		\widetilde{g}\,:\,\mathcal {PN}^{\alpha}(r) \,\longrightarrow \, PN^{\alpha}(r). 
	\end{align}
	As before, the connected components of $\mathcal {PN}^{\alpha}(r)$ and $PN^{\alpha}(r)$ are
	indexed by integers $i \,\in\, [0,\,r-1]$; the $i$-th connected component of
	$\mathcal {PN}^{\alpha}(r)$ (respectively, $PN^{\alpha}(r)$) will be
	denoted by $\mathcal {PN}^{\alpha}(r)_{i}$ (respectively, $PN^{\alpha}(r)_{i}$).
	The scheme $PN^{\alpha}(r)_{i}$ is the quotient of $M^{\alpha}_{\xi}(r,d)$ by
	the above mentioned group $\Gamma$, while the stack $\mathcal{PN}^{\alpha}(r)_i$ is the
	quotient of the stack $\mathcal {M}^{\alpha}_{\xi}(r,d)$ by $\rm Bun_{\mu_r}$,
	where $i \,\equiv\, d\,\ (\textnormal{mod}\ r)$.
	In both cases, the action is given by parabolic tensor product with line bundles
	equipped with the trivial parabolic
	structure. The morphism $\widetilde{g}$ in \eqref{mapping1} sends
$\mathcal{PN}^{\alpha}(r)_{i}$ onto $PN^{\alpha}(r)_{i}$. 
	
	Let $PN^{\alpha,sm}(r)_i$ denote the smooth locus of $PN^\alpha(r)_i$. Our aim here is to study the Brauer group of the spaces $\mathcal {PN}^{\alpha}(r)_{i}$ and $PN^{\alpha,sm}(r)_{i}$. We prove that the two Brauer groups in question coincide, 
	and is in fact isomorphic to the kernel in \eqref{ei}.
	
	The main results are as follows:
	
	\begin{theorem}\label{thi}\mbox{}
		Let $\alpha$ be a generic parabolic weight, meaning every parabolic semistable
bundle with weight $\alpha$ is parabolic stable.
		\begin{enumerate}
			\item In the case of full flags (i.e., $r_{i,j}\,=\, 1\,\,\forall\,(i,\,j)$), the Brauer groups of $\mathcal {PN}^{\alpha}(r)_{i}$  and $PN^{\alpha,sm}(r)_{i}$ coincide.
			
\item Let $m$ be as in \eqref{gcd}. If $m$ equals 1, then
$$\textnormal{Br}(PN^{\alpha,sm}(r)_{i})\,\cong\, \frac{H^2(\Gamma,\, k^{*})}{H},$$ where
$H \,\subset\, H^2(\Gamma,\, k^{*})$ is as in \eqref{ei}.
		\end{enumerate}
	\end{theorem}
	
	Theorem \ref{thi} $(1)$ is proved by producing an open subscheme $U$ of 
	$\textnormal{PN}^{\alpha}(r)$ whose complement is of codimension at least 3 such that the map $\widetilde{g}$ defined in (\ref{mapping1}) is an isomorphism over $U$ (cf. Corollary \ref{isomorphism of brauer group for stack}). For Theorem \ref{thi} $(2)$, we first assume the parabolic weights are 
	sufficiently small, so that there exists a forgetful map from an open subscheme 
	$\mathcal{U}_\alpha\,\subset\, M^{\alpha}_{\xi}(r,d)$ to $M_{\xi}(r,d)$ by simply
	forgetting the parabolic structure. This map is an \'etale-locally trivial fibration. Next, 
	we restrict our attention over a suitably chosen open subscheme of $M^{\alpha}_{\xi}(r,d)$ 
	on which $\Gamma$ acts freely. Quotienting by $\Gamma$ then gives rise to a finite \'etale 
	morphism, from which we obtain our result using the Hochschild-Serre spectral sequence 
	corresponding to this finite \'etale map (cf. Proposition \ref{brauer group iso 1}). Finally, for arbitrary generic weights, we use a result of Thaddeus 
	\cite[\S~6, 6.2]{Th02} which says that the moduli spaces with different weights are actually birational. \\
	Finally, we discuss the case when the weights are not generic, but still satisfy condition $m=1$ where $m$ is as in (\ref{gcd}). We show that for a non-generic weight $\gamma$ which lie on a
single \textit{wall} \cite[\S~2]{BY99}, the Brauer group of the smooth locus of the corresponding moduli
of parabolic stable ${\rm PGL}(r)$-bundles, namely $\Br(PN^{\gamma,sm}(r)_i)$, is isomorphic to $\textnormal{Br}(PN^{\alpha,sm}(r)_{i})$ for any generic weight $\alpha$. This is proved in Proposition \ref{non-generic weight proposition}.

Let us also mention that in case when $k=\mathbb{C}$ and the curve $X=X'\otimes_{\mathbb{R}}\mathbb{C}$ for a geometrically smooth projective surve $X'$ of genus $g\geq 2$ over $\mathbb{R}$, then the Brauer group of the moduli of geometrically stable vector bundles of fixed rank and determinant has been computed in \cite{BHHS11}. Very recently, the Brauer group of the moduli stack (and moduli space) of geometrically stable parabolic bundles of fixed determinant over a geometrically irreducible projective curve with at most nodes as singularities has been computed in \cite{BB23}.

	\section{Preliminaries}\label{se2}
	
	Let $X$ be an irreducible smooth projective curve of genus $g\,\geq\, 2$,
	defined over an algebraically closed field $k$ of characteristic zero.
	A vector bundle on $X$ will always mean an algebraic vector bundle.
	Fix an integer $r\,\geq\, 2;$ if $g\,=\,2,$
	then set $r\,\geq\, 3$. Fix a line bundle $\xi$ on $X$ of degree $d$. By a point of $X$ we will always mean a closed point. Fix a subset
	$$S\,=\,\{p_1,\,p_2,\,\cdots,\,p_s\}\, \subset\, X$$ of $s$ distinct points; they will be referred to
	as \textit{parabolic points}.
	Let $E$ be a vector bundle of rank $r$ on $X$. The fiber of $E$ over a point $z\, \in\, X$ will be denoted by $E_z$. 
	
	\begin{definition}\label{pardef}
		A \textit{parabolic data of rank r} for points of $S$ consists of the following collection: for each $p_i\,\in \,S$
		\begin{itemize}
			\item a string of positive integers $(m^i_1,\,m^i_2,\,\cdots,\,m^i_{l_i})$ such that $\sum_{j=1}^{l_i}m^i_j \,=\, r$, and
			
			\item an increasing sequence of real numbers $0\,\leq\, \alpha^i_1\,<\,\alpha^i_2\,<\,\cdots\,<\,\alpha^i_{l_i}\,<\,1$.
		\end{itemize}
		
		If $m^i_j \,=\, 1$ for all $1\,\leq\, j\,\leq \,l_i$ and $1\,\leq \,i\,\leq \,s$, we say that it is a \textit{full-flag} parabolic data.
		
		A \textit{parabolic structure} on $E$ over $S$, with parabolic data of rank $r$ as above, consists of the following:
		for each $p_i\,\in\, S$, a weighted filtration
		\begin{align*}
			E_{p_i} \,=\, E^i_1\,\supsetneq\, E^i_2\,\supsetneq\, &\,\cdots\, \supsetneq\, E^i_{l_i}\,\supsetneq\, E^i_{l_i+1} \,=\, 0\\
			0\,\leq\, \alpha^i_1\,<\,\alpha^i_2\,<\,&\,\cdots\,<\,\alpha^i_{l_i}\,<\,\alpha^i_{l_i+1} \,=\, 1
		\end{align*}
		such that $m^i_j \,=\, \dim(E^i_j/E^i_{j+1})\,\,\,\,\forall\,\,\,1\,\leq\, j\,\leq\, l_i, \,1\,\leq \,i\,\leq\, s$.
		
		The above collection 
		$\alpha\,:=\,\{(\alpha^i_1\,<\,\alpha^i_2\,<\,\cdots\,<\,\alpha^i_{l_i})_{ 1\leq i\leq 
			s}\}$ is called the set of \textit{weights}, and the above integer $m^i_j$ is called the 
		\textit{multiplicity} of the weight $\alpha^i_j$.
		
		Let $\textbf{m} = \{(m^i_1,\cdots,m^i_{l_i})_{1\leq i\leq s}\}$. By a \textit{parabolic bundle} we mean a collection of data $(E,\, \textbf{m},\, \alpha)$, where $E$ is a vector bundle on $X$, while $\textbf{m}$ and $\alpha$ are as 
		described above. For notational convenience, such a parabolic vector
		bundle will also be referred to as $E_*$; the vector bundle
		$E$ is called the underlying bundle. A parabolic structure with	a full-flag parabolic data is called a full-flag parabolic structure.
	\end{definition}
	
	\begin{definition}
		Let $E_*$ be a parabolic bundle on $X$ of rank $r$. The \textit{parabolic degree} of $E_*$ is defined as 
		\[\textnormal{pardeg}(E_*)\,:=\, deg(E)+\sum_{i=1}^{s}\sum_{j=1}^{l_i}m^i_j\alpha^i_j, \] 
		and the \textit{parabolic slope} of $E_*$ is defined as 
		\[\textnormal{par}\mu(E_*)\,:=\, \dfrac{\textnormal{pardeg}(E_*)}{r}.\]
		The parabolic bundle $E_*$ is called \textit{parabolic semistable} (respectively,
		\textit{parabolic stable}) if for every proper sub-bundle $0\, \not=\, F\, \subset\, E$ we have 
		\[
		\textnormal{par}\mu(F_*)\,\leq\, \textnormal{par}\mu(E_*)\ \ \, (\text{respectively, }\
		\textnormal{par}\mu(F_*)\, <\, \textnormal{par}\mu(E_*)),\]
		where $F_*$ denotes the parabolic bundle defined by $F$ equipped with the induced parabolic structure from $E_*$ (see \cite{MS80} for
		the details).
	\end{definition}

See \cite{Ses}, \cite{MS80} for homomorphisms of parabolic bundles. The class of parabolic semistable bundles 
with fixed parabolic slope forms an abelian category. We refer to \cite[p.~ 68]{Ses} for further details. In 
particular, the direct sum of parabolic semistable bundles with the same parabolic 
slope is again semistable of the same slope.

\begin{definition}
A system of weights $\alpha$ is called \textit{generic} if every semistable parabolic bundle $E_*$, of given rank and
degree, with weights $\alpha$ is actually parabolic stable. We refer to \cite{BY99} for more details.
\end{definition}
	
	\subsection{Parabolic push-forward and pull-back}\label{parabolic pushforward}
	
	Let $X$ and $Y$ be two irreducible smooth projective curves, and let $\gamma\,:\, 
	Y\,\longrightarrow\, X$ be a finite \'etale Galois morphism. If $F$ is a vector bundle on 
	$Y$ of rank $n$, then $\gamma_*F$ is a vector bundle on $X$ of rank $mn$, where $m$ is the 
	degree of the map $\gamma$. Given a parabolic structure on $F$, there is a natural way to 
	construct a parabolic structure on $\gamma_*F$. We refer to \cite[\S~3]{BM19}
for details (a more general construction can be found in \cite{AB}).
	
	Let us mention a special case of parabolic push-forward which will be used here. Let $p\,\in\, X$ be a point. Let $m$ be the degree of $\gamma$. Since $\gamma$ is unramified, the inverse image $\gamma^{-1}(p)$ consists of $m$ distinct points of $Y$.  Suppose we are given a full-flag parabolic data of rank $n$ with $\gamma^{-1}(p)$ as the set of parabolic points, so that 
	the parabolic data looks like
	\begin{align*}
		q\,\in\, \gamma^{-1}(p),\,\,F_q \,=\, F^q_1\,&\supsetneq\, F^q_2\,\supsetneq \,\cdots\,\supsetneq\, F^q_n
		\,\supsetneq\, F^q_{n+1}\,=\,0,\\
		\alpha^q_1\,&<\,\alpha^q_2\,<\,\cdots\cdots<\,\alpha^q_n\,<\,\alpha^q_{n+1}\,=\,1.
	\end{align*}
	Moreover, we assume that in the collection $\{\alpha^q_j\,\mid\,\,q\,\in\,\gamma^{-1}(p),\, \,1\,\leq\, j\,\leq\, n\}$
	all numbers are distinct. Let $F_*$ be a parabolic vector bundle on $Y$ with this parabolic data. 
	We shall construct a full-flag parabolic structure for $E\,=\,\gamma_*F$ at the point $p$ from this data. Note that
	\[E_p\,=\,\bigoplus_{q\,\in\, \gamma^{-1}(p)}F_{q}.\]
	Let $\beta^p_1\,<\,\beta^p_2\,<\,\cdots\,<\,\beta^p_{mn}$ be the increasing sequence of length $mn$
	obtained by ordering the numbers $\{\alpha^q_j\,\mid\,\,q\,\in\,\gamma^{-1}(p),\,1\,\leq\, j\,\leq\, n\}$. For each integer
	$1\,\leq\, k\,\leq\, mn$, define a filtration of $E_p$ by the subspaces
	$$E^p_k \,:=\, \bigoplus_{q\in\gamma^{-1}(p)}F^q_{\omega(q,k)},$$
	where $\omega(q,k)$, for each point $q$, is the smallest integer $1\,\leq\, j(q)\,\leq\, n$ satisfying the
	condition $\beta^p_k\,\leq\, \alpha^q_{j(q)}.$
	It is straight-forward to see that $\dim(E^p_k/E^p_{k+1}) \,=\, 1\,\ \,\forall\ \,1\,\leq\,
	k\, \leq\, mn-1$ and $\dim E^p_{mn}\,=\, 1$. Consequently,
	\begin{align*}
		E_p \,=\, E^p_1&\,\supsetneq\, E^p_2\,\supsetneq\,\cdots\,\supsetneq\, E^p_{mn}\\
		\beta^p_1&\,<\,\beta^p_2\,<\,\cdots\cdots\,<\,\beta^p_{mn}
	\end{align*}
	is a full-flag parabolic structure of rank $mn$ at $p$.
	
	Finally, for multiple parabolic points on $X$, we perform exactly the same construction 
	for each parabolic point to define the parabolic push-forward.
	
	Let $\gamma\,:\,Y\,\longrightarrow\, X$ be as above. If $E_*$ is a parabolic bundle on $X$, then $\gamma^*E$ has an induced parabolic structure as follows: let $S\subset X$ be the set of parabolic points for $E$; we define $\gamma^{-1}(S)$ as the set of parabolic points for $\gamma^*E$, and for any $p\in S$ and $q\in \gamma^{-1}(p)$, give the fiber $(\gamma^*E)_q=E_p$ the same parabolic structure as $E_p$. 
	
	\subsection{Moduli of parabolic ${\rm PGL}(n,k)$-bundles}
	
	\begin{definition}
		A parabolic $\textnormal{PGL}(n,k)$-bundle is an equivalence class of parabolic vector bundles, where two parabolic bundles $E_*$ and $E'_*$ 
		are considered equivalent if there exists a line bundle $L$ such that $E'_*\,\simeq\, E_*\otimes L$ as parabolic bundles.
	\end{definition}
	
	Let $\mathcal{M}^{\alpha}_{\xi}(r,d)$ denote the moduli stack of stable parabolic bundles on $X$ of
rank $r$ and fixed determinant
	$\xi$ of degree $d$ with parabolic data $\alpha$; the corresponding coarse moduli space will be denoted by
	$M^{\alpha}_{\xi}(r,d)$. The natural morphism 
	\[
	\rho\,:\, \,\mathcal{M}^{\alpha}_{\xi}(r,d)\,\longrightarrow\, M^{\alpha}_{\xi}(r,d)
	\]
	is a $\mathbb G_m$-gerbe; this follows from the fact that the automorphisms of
	a stable parabolic bundle are the nonzero scalar multiplications.
	
	Let
	$$\Gamma\,\subset\, {\rm Pic}^0(X)$$
	be the group of isomorphism classes of line bundles on $X$ of order $r$; in other words, $\Gamma$ is the
	$r$--torsion points of the Jacobian of $X$. This group $\Gamma$ acts naturally
	on $PN^{\alpha}_{\xi}(r,d)$ by parabolic tensor product. More precisely, the action
	of any $L\,\in\, \Gamma$ sends any $E_*\,\in\, PN^{\alpha}_{\xi}(r,d)$ to the parabolic
	tensor product
	\begin{equation}\label{eac}
		L\cdot (E_*) \,:=\, E_*\otimes L,
	\end{equation}
	where $L$ has the trivial parabolic structure.
	
	Let ${\rm Bun}_{\mu_r}$ denote the group stack of $\mu_r$-bundles on $X$. The natural map 
	$$h\,:\,{\rm Bun}_{\mu_r}\,\longrightarrow \,\Gamma$$ is a $\mu_r$-gerbe. The tensor product
	operation  defines an action of the stack $\rm Bun_{\mu_r}$ on
	$\mathcal{M}^{\alpha}_{\xi}(r,d)$. This action is
	clearly compatible with the $\Gamma$-action on $M^{\alpha}_{\xi}(r,d)$, 
	meaning the following diagram commutes:
	\[\xymatrix{
		\rm Bun_{\mu_r}\times \mathcal{M}^{\alpha}_{\xi}(r,d) \ar[r]\ar[d]_{h\times\rho}
		& \mathcal{M}^{\alpha}_{\xi}(r,d) \ar[d]^{\rho} \\
		\Gamma\times M^{\alpha}_{\xi}(r,d) \ar[r] & M^{\alpha}_{\xi}(r,d)
	} \] 
	Let $i\,\equiv\, d\ \,(\textnormal{mod r})$\, with $i\,\in\, [0,\,d-1]$. The moduli stack of 
	parabolic stable $\textnormal{PGL}(r,k)$-bundles on $X$ of topological type $i$, denoted 
	by $\mathcal{PN}^{\alpha}(r)_i$, is defined to be the quotient of 
	$\mathcal{M}^{\alpha}_{\xi}(r,d)$ by the action of $\rm Bun_{\mu_r}$.\\ The moduli space 
	of parabolic stable $\textnormal{PGL}(r,k)$-bundles on $X$ of topological type $i$ is 
	defined similarly as the quotient of $M^{\alpha}_{\xi}(r,d)$ by $\Gamma$:
	\[PN^{\alpha}(r)_i \,:=\, M^{\alpha}_{\xi}(r,d)/\Gamma\,.  \] 
	We have the following commutative diagram 
	\[
	\xymatrix{
		\mathcal{M}^{\alpha}_{\xi}(r,d) \ar[r]^{\rho} \ar[d] & M^{\alpha}_{\xi}(r,d) \ar[d] \\
		\mathcal {PN}^{\alpha}(r)_i \ar[r]^{\overline\rho} & PN^{\alpha}(r)_i.
	}	
	\]	
	See \cite[\S~2]{GO18} for more details.
	
	\section{The fixed point loci}
	
	We fix a positive integer $r$, a subset of parabolic points of $X$
	and a line bundle $\xi$ of degree $d$ on $X$. Assume that $i\,\equiv\, d\,\ (\text{mod\, r}),\,i\,\in\,[0,\,d-1]$ in what follows. We adopt the following notations:
	\begin{align*}
		\text{M}_{\xi}^{ss}  := &\,\,\text{coarse moduli space of semistable vector bundles of rank}\,\, r\,\,\text{and determinant iso-}\\
		&\text{morphic to}\,\,\xi.	\\
		\Mxi := &\,\,\text{moduli space of stable vector bundles of rank}\,\, r\,\,\text{and determinant isomorphic to}\,\,\xi.	\\
		\text{N}(r)_i := &\,\,\text{moduli space of stable}\,\,\text{PGL}(r,k)\text{--bundles of rank}\,\,r\,\,\text{and topological type} \,\,i,\\
		\Malpha := &\,\,\text{moduli space of stable parabolic vector bundles of rank}\,\,r,\,\,\text{fixed determinant}\,\,\xi\\
		&\,\,\text{with full flag and weights}\,\,\alpha.\\
		\text{PN}_{\alpha,i}:= &\,\,\text{moduli space of stable parabolic}\,\text{PGL}(r,k)\text{--bundles of rank}\,\,r,\,\,\text{topological type}\\
		&\,i\,\,\text{with full flag and weights}\,\,\alpha.\\
		\mathcal{PN}_{\alpha,i} := &\,\,\text{moduli stack of stable parabolic}\,\text{PGL}(r,k)\text{--bundles of rank}\,\,r,\,\,\text{topological type}\\
		&\,i\,\,\text{with full flag and weights}\,\,\alpha.\\
		\Gamma := &\,\,\{L\,\in \,\text{Pic}^0(X)\,\,\mid\,\,L^r\,\simeq\, \mathcal{O}_X\}\simeq (\mathbb{Z}/r\mathbb{Z})^{2g}.
	\end{align*}
	The aim in this section is to estimate the codimension of the fixed point locus of $\Malpha$
	under any $r$-torsion line bundle $L$. This would ensure the existence of a closed subscheme
	of $\Malpha$ (namely, the union of the fixed point locus for all $L\,\in\, \Gamma\setminus\{\mathcal{O}_X\}$) of codimension at least three such that $\Gamma$ acts freely on its complement. We remark that $\alpha$ need not be a generic weight.
	We first prove a result in linear algebra which will be used.
	
	\begin{lemma}\label{linear algebra result}
		Let $\psi$ be a diagonalizable automorphism of a
		$k$-vector space $V$ of dimension $r$ equipped with a filtration of subspaces
		\[V\,=\,V_r\,\supsetneq\, V_{r-1}\,\supsetneq\, V_{r-2}\,\supsetneq\, \cdots\,
		\supsetneq\, V_1\,\supsetneq\, 0\]
		such that $\psi(V_i)\,=\, V_i$ for all $1\, \leq\, i\, \leq\, r$.
		Then there exists a basis of $V$ consisting of eigenvectors
		$\{v_1,\,v_2,\,\cdots,\,v_r\}$ of $\psi$ such that
		$V_j\,=\,\langle v_1,\,\cdots,\, v_j\rangle$ for all $1\, \leq\, j\, \leq\, r$.
	\end{lemma}
	
	\begin{proof}
		Since $\psi$ is diagonalizable, for any subspace $W\, \subset\, V$ such that
		$\psi(W)\,=\, W$, the restriction of $\psi$ to $W$ is diagonalizable and, moreover, since diagonalizable maps are semisimple, there
		is a subspace $W'\, \subset\, V$ satisfying the conditions $\psi(W')\,=\, W'$ and
		$V\,=\, W\bigoplus W'$. Choose any basis vector $v_1$ for $V_1$. Suppose $\{v_1,v_2,\cdots,v_j\}$ has been chosen satisfying the hypothesis up to $j\leq r-1$. Thus there exists a vector $v_{j+1}\in V_{j+1}$ such that $V_{j+1}=\langle v_{j+1}\rangle \bigoplus V_j$ and $v_{j+1}$ is an eigenvector of $\psi$. Then $\{v_1,\cdots,v_{j+1}\}$ satisfy the hypothesis till $(j+1)$. Repeating this process, the lemma follows.
	\end{proof}
	
	For $L\,\in \,\Gamma\setminus\mathcal{O}_X$, let
\begin{equation}\label{efp}
\MLxi\,\subset\,\Mxi\ \ \ {\rm (respectively,}\ \, \MLalpha\,\subset\,\Malpha{\rm )}
\end{equation}
be the locus of fixed points for the action of $L$
	on $\Mxi$ (respectively, $\Malpha$); see \eqref{eac}. 
	If $m\,=\,\text{ord}(L)$, choosing a nonzero section $s_0\, \in\, H^0(X,\, L^{\otimes m})$,
	define the spectral curve
	$$
	Y_L\, :=\, \{v\, \in\, L\,\, \mid\,\, v^{\otimes m}\, \in\, s_0(X)\}\, .
	$$
	The natural projection $\gamma\,:\,Y_L\,\longrightarrow\, X$ is an \'etale Galois covering
	with Galois group ${\mathbb Z}/m{\mathbb Z}$. The isomorphism class of this
	covering $\gamma$ does not depend on the choice of the section $s_0$.
	Let $N_L\,\subset\, M_{Y_L}(r/m,d)$ denote the moduli space of stable vector bundles $F$ over	$Y_L$ of rank $r/m$ and degree $d$ such that $\det (\gamma_*F)\,\simeq\, \xi$. We recall the following:

	\begin{lemma}[\text{\cite[Lemma 2.1]{BH10}}]
		There is a nonempty Zariski open subset $U\,\subset\, N_L$ such that $\gamma_*F\,\in\, \MLxi$
		for all $F\,\in\, U$. Moreover, the morphism $U\,\longrightarrow\, \MLxi$ defined by
		$F\,\longmapsto \,\gamma_*F$ is surjective. To be more precise, $U$ is the subset of
		$N_L$ consisting of those bundles $F$ such that $\gamma_*F$ is also stable.
	\end{lemma}
	
	\subsection{Estimate of codimension of fixed points}
	
	Given a set $\alpha\,=\,\{\alpha_1,\,\alpha_2,\,\cdots,\, \alpha_r\}$ of $r$ distinct 
	elements, let $P(\alpha)$ denote the set of all possible partitions of its elements into 
	$m$ subsets, each containing $n\,=\,r/m$ elements. Clearly, we have $|P(\alpha)|\,=\,
	{r\choose n}{r-n \choose n}{r-2n\choose n}\cdots {n\choose n}.$
	
	Consider the spectral curve $\gamma\,:\, Y_L\,\longrightarrow \,X$. Given a full-flag 
	parabolic data of rank $r$ at the parabolic points $S$ of $X$, we would like to describe a 
	full-flag parabolic data of rank $n$ on $\gamma^{-1}(S)$ for each element of $P(\alpha)$. 
	For simplicity, let us start with the case of a single parabolic point $S\,=\,\{p\}$. So, 
	we are given a full-flag parabolic data of rank $r$ at $p$. Let $\alpha$ denote its set of 
	weights. Let $$\textnormal{Gal}(\gamma)\,=\,\{1, 
	\,\mu,\,\mu^2,\,\cdots,\,\mu^{m-1}\}\,\subset \,k^*.$$ The action of 
	$\textnormal{Gal}(\gamma)$ on $\gamma^{-1}(p)$ is via multiplication. Using this, we can 
	define a full-flag parabolic data at the points of $\gamma^{-1}(p)$ as follows: let us fix 
	an ordering on the points of $\gamma^{-1}(p)$, say 
	$\gamma^{-1}(p)\,=\,\{q_1,\,q_2,\,\cdots,\,q_m\}$, such that $\mu^i$ acts on 
	$\gamma^{-1}(p)$ as the cyclic permutation sending $q_j$ to $q_{j+i}$\,, where the 
	subscript $(j+i)$ is to be understood mod $m$. For $t\,\in\, P(\alpha)$, suppose \[\alpha 
	\,=\, \coprod_{j=1}^{m} \Lambda_j \] be the partition of the set of weights $\alpha$ 
	according to $t$. Clearly each $\Lambda_j$ can be arranged into an increasing sequence of 
	length $n$. We designate $\Lambda_j$ as the set of weights at $q_j$ for each $1\,\leq\, 
	j\,\leq\, m$. This gives a full-flag parabolic data of rank $n$ at the points of 
	$\gamma^{-1}(p)$. Finally, for multiple parabolic points, say $S\,= 
	\,\{p_1,\,\cdots,\,p_s\}$ and $\gamma^{-1}(p_i)\,=\,\{q_{1i},\,q_{2i},\,\cdots,q_{mi}\}$, 
	then we perform the same procedure as above for each $p_i$ to get the parabolic structure 
	upstairs. In particular, the number of possible parabolic data upstairs would be 
	$|P(\alpha)|\, = \, ({r\choose n}{r-n \choose n}{r-2n\choose n}\cdots {n\choose n})^s.$
	
	For each $t\,\in\, P(\alpha)$, let $M^t_{Y_L}(n,d)$ denote the moduli space of stable parabolic 
	bundles over $Y_L$ of rank $n$, degree $d$ and having full-flag parabolic structures 
	at the points of $\gamma^{-1}(p)$ according to $t$ as described above. Let 
	$\mathcal{N}^t_L\,\subset\, M_{Y_L}^{t}(n,d)$ denote the subvariety consisting of stable
	parabolic bundles $F_*$ such that $\det(\gamma_*F)\,\simeq\, \xi$. Define 
	$\mathcal{N}_L\,:=\,\underset{t\in P(\alpha)}{\coprod}\mathcal{N}^t_L.$
	
	\begin{lemma}\label{parabolic fixed point}
		There is a surjective morphism $f\,:\,\mathcal{N}_L\,\longrightarrow\, \MLalpha$ given by
		parabolic push-forward for the \'etale map $\gamma$.
	\end{lemma}
	
	\begin{proof}
		First, for simplicity let us assume that we have only one parabolic point $S\,=\,\{p\}$. The 
		map $f$ in the statement of the lemma sends any $F_*$
		to the parabolic push-forward of $F_*$ constructed in Section 
		\ref{parabolic pushforward}.
		
		We claim that $f(F_*)$ is parabolic semistable. To prove this, note
		that if $E_*\,:= \,f(F_*)$, then $$\gamma^*(E_*)\,\cong\, \bigoplus_{\sigma\in 
			\textnormal{Gal}(\gamma)} \sigma^*(F_*),$$ where $\sigma^*(F_*)$ has the obvious parabolic 
		structure coming from $F_*$. Clearly $\text{par}\mu(\sigma^*(F_*))\,=\,\text{par}\mu(F_*)$ for all $\sigma$. Therefore 
		$\gamma^*(E_*)$ is a direct sum of parabolic stable bundles of same parabolic slope, which 
		implies that $\gamma^*(E_*)$ is parabolic semistable. Thus $E_*$ must be parabolic 
		semistable as well, since any sub-bundle $E'_*\,\subset\, E_*$ with strictly larger parabolic 
		slope would give rise to a subbundle $\gamma^*(E'_*)\,\subset \,\gamma^*(E_*)$ of strictly 
		larger parabolic slope, contradicting parabolic semistability of $\gamma^*(E_*)$.
		Hence $f(F_*)$ is parabolic semistable.
		
		It can be shown that $E_*$ is 
		actually parabolic stable. To prove this, take any non-trivial sub-bundle $E'\,\subset\, E$ such 
		that $\textnormal{par}\mu(E'_*)\,= \,\textnormal{par}\mu(E_*)$. Then $$\gamma^*(E'_*)\,\subset 
		\,\gamma^*(E_*)\,=\,\underset{\sigma\in \textnormal{Gal}(\gamma)}{\bigoplus} \sigma^*(F_*)$$
		is a sub-bundle with the same parabolic slope. Now, it is clear that $\gamma^*(E'_*)$ is also 
		parabolic semistable. Thus it contains a parabolic stable sub-bundle $F'_*$. We note that the 
		sub-bundles $$\{\sigma^*(F_*)\,\,\mid\,\,\sigma\in\textnormal{Gal}(\gamma)\}$$
		are mutually non-isomorphic. Indeed, otherwise there would exist a parabolic
		isomorphism $\widetilde f$ from $F_*$ to $\sigma^*F_*$ for some
		$\sigma\,\in\, {\rm Gal}(\gamma)\setminus\{e\}$; at a parabolic point $y$ it should preserve the
		filtrations of $F_y$ and $F_{\sigma(y)}$. But since the weights of $F_*$ at the parabolic
		points $y$ and $\sigma(y) $ are collectively distinct, and a parabolic
		isomorphism preserves weights, this leads to a contradiction.
		Therefore, it follows that all projections $F'_*\,\longrightarrow\, \sigma^*(F_*)$ except one must 
		be zero, so that $\gamma^*(E'_*)$ contains one of the $\sigma^*(F_*)$. But $\gamma^*(E'_*)$ 
		is $\textnormal{Gal}(\gamma)$-equivariant, which would force that $\gamma^*(E'_*) \,=\, 
		\gamma^*(E_*)$; this clearly implies that $E'_*\,=\,E_*$.
		Consequently, $E_*$ is parabolic stable.
		
		Next, we argue that $\textnormal{Im}(f)\,\subseteq\,\MLalpha$. Let $E_*\,=\,f(F_*)$. There exists a tautological 
		trivialization of $\gamma^*L$ over $Y_L$, which induces an isomorphism 
		\begin{align}\label{tautological trivialization}
			\theta\,:\, \mathcal{O}_{Y_L}\,\longrightarrow\, \gamma^*L\,.
		\end{align}  For each $1\leq i\leq m$, the induced map $\theta_{q_i}\,:\, k\,
\longrightarrow\, (\gamma^*L)_{q_i}\,=\,L_p$ is defined by $\lambda \,\longmapsto\, \lambda q_i.$ This induces an isomorphism
		\begin{align*}
			Id_F\otimes \theta \,:\, F\,\xrightarrow{\,\,\,\simeq\,\,}\, F\otimes \gamma^*L.
		\end{align*}
		This, in turn, produces a canonical isomorphism $\psi: E\simeq E\otimes L$ on the underlying bundle arising from $\gamma_*(Id_F\otimes\theta)$ followed by projection formula:
		\[ \psi\,:\,E\,=\,\gamma_*F\,\xrightarrow{\,\,\,\simeq\,\,}\,{\gamma_*(Id_F\otimes\theta)}
\, \gamma_*(F\otimes \gamma^*L)\,\cong\, E\otimes L.\] 
		Now $E_p\,=\, \bigoplus_{i=1}^{m}F_{q_i}$, and the map $\psi_p\,:\,E_p
\,\longrightarrow\, E_p\otimes L_p$ on the fiber takes $F_{q_i}$ to $F_{q_i}\otimes L_p$, which clearly implies that 
		$\psi_p$ preserves the filtration induced on $E_p$. Thus we have $\text{Im}(f)\,\subseteq\, \MLalpha.$
		
		To show the surjectivity of $f$, let $E_*\,\in\, \MLalpha$, so there exists a parabolic isomorphism 
		\[\varphi_*\,:\, E_*\,\stackrel{\simeq}{\longrightarrow}\, E_*\otimes L.\] Suppose 
		$(\alpha^p_1\,<\,\alpha^p_2\,<\,\cdots\,<\,\alpha^p_r)$ are the weights at $p$. We would like
		to show 
		that there exists $F_*\,\in\, \mathcal{N}_L$ such that $f(F_*)\,\simeq\, E_*$. Since $E_*$ is 
		parabolic stable, it is simple, and hence any parabolic endomorphism of $E_*$ is 
		a constant scalar multiplication. As a consequence, any two parabolic isomorphisms from 
		$E_*$ to $E_*\otimes L$ will differ by a constant scalar multiplication. Thus, we can re-scale $\varphi_*$ 
		by multiplying with a nonzero scalar, so that the $m$--fold composition
		\[\underset{m-times}{\varphi_*\,\circ\cdots\circ\,\varphi_*}\,:\, E_*
		\,\longrightarrow\, E_*\otimes L^m\]
		coincides with $\textnormal{Id}_{E_*}\otimes \tau $, where $\tau$ is the chosen
		nowhere vanishing section of $L^m$. This produces the isomorphism
		\[\underset{m-times}{\varphi\,\circ\cdots\circ\,\varphi} \,=\,
		\textnormal{Id}_E\otimes \tau\,:\, E \,\longrightarrow\, E\otimes L^m\]
		on the underlying bundles. Then, the argument given in the proof of \cite[Lemma 2.1]{BH10}
		will produce a vector bundle $F$ on $Y_L$ with $\gamma_*F\,\cong\, E$. Let us briefly recall
		the argument \cite{BH10} for convenience. Consider the pull-back $\gamma^*\varphi$, and compose it with the tautological trivialization of $\gamma^*L$ to get a morphism 
		\[\phi\,:\, \gamma^*E\,\longrightarrow\, \gamma^*E.\] 
		Since $Y_L$ is irreducible, the characteristic polynomial of $\phi_y$ remains unchanged as
		$y\,\in \,Y_L$ moves. This allows us to decompose $\gamma^*E$ into generalized eigenspace
		sub-bundles. If $F$ is an eigenspace sub-bundle of $\gamma^*E$, then we have
		$\gamma_*F\,\cong\, E$. Moreover, the decomposition $\gamma^*(E)=\bigoplus_{i=1}^{m}(\mu^i)^*F$ is precisely the decomposition of $\gamma^*(E)$ into generalized eigenspace sub-bundles.
		
		Our next task is to produce a parabolic structure on $F$ so that the parabolic push-forward $\gamma_*(F_*)$ (in the sense of Remark \ref{parabolic pushforward}) coincides with $E_*$. To see this, 
		recall the description of $\theta$ in (\ref{tautological trivialization}), and notice that for any choice of $q\in \gamma^{-1}(p)$, the map $\phi_q$ is precisely the composition
		\begin{align}\label{phi_q}
			(\gamma^*E)_q=E_p\xrightarrow{\varphi_p} E_p\otimes L_p = (\gamma^*E)_q\otimes (\gamma^*L)_q\xrightarrow{Id\otimes(\theta_q)^{-1}}(\gamma^*E)_q\,,\,
		\end{align}
		where $\theta_q\,:\, k\,\longrightarrow\, (\gamma^*L)_q\,=\,L_p$ is defined by
$\lambda\,\longmapsto\, \lambda q$. Thus, if
		\begin{align*}
			E_p\,=\,E^p_1\,\supsetneq \,E^p_2\,\supsetneq\,\cdots\,\supsetneq\,E^p_r\,\supsetneq\,0
		\end{align*}
		is the given parabolic filtration of the fiber $E_p$, then as $\varphi_*$ is a parabolic isomorphism,
		\begin{align}
			\forall\,j\in[0,r],\,\,\,\varphi_p(E^p_j)\,=\,E^p_j\otimes L_p\,\implies\,\phi_q(E^p_j)\,=\,E^p_j\,\,\,\,[\text{from}\,(\ref{phi_q})].
		\end{align}
		Let $\phi^s_q$ be the semisimple part of $\phi_q$ under its Jordan-Chevalley decomposition. It is well-known that $\phi^s_q$ can be expressed as a polynomial in $\phi_q$ without constant coefficient. Thus 
		\[\phi^s_q(E^p_j) = E^p_j\,\,\,\forall\,j\in[0,r-1].\]
		Moreover, the generalized eigenspaces of $\phi_q$ (namely $F_{q_i}$'s) are the eigenspaces for $\phi^s_q$.
		Thus $\phi^s_q\,:\,E_p\,\longrightarrow\, E_p$ and the filtration $E_p\,=\,E^p_1
\,\supsetneq\,\cdots\,\supsetneq\, E^p_r\,\supsetneq\,0$ allow us to apply Lemma \ref{linear algebra result}, which gives us a basis of $E_p$\,\,, say $\{v_1,\cdots,v_r\}$, such that each $E^k_p = \langle v_k,\cdots, v_r\rangle$ and each $v_j$ is contained in a unique $F_{q_l}$. From this data, we can produce a full-flag parabolic structure on the
		fibers $F_{q_1},\cdots,F_{q_m}$ of $F$ as follows:
		
		Choose a basis $B\,=\,\{v_1,\,v_2,\,\cdots,\,v_r\}$ of $E_p$ as in
		Lemma \ref{linear algebra result}, and for each $q_i$, consider the set $B_i\,:=
		\,B\cap F_{q_i}$. By symmetry, each $B_i$ consists of $n$ elements and spans $F_{q_i}$.
		Suppose that $B_i \,=\, \{v_{i_1},\,v_{i_2},\,\cdots,\,v_{i_n}\},$ where
		$i_1\,<\,i_2\,<\,\cdots\,<\,i_n.$ Then consider the following weighted full-flag filtration
		at $F_{q_i}$:
		\begin{align*}
			F_{q_i}\,=\,\langle v_{i_1},\,v_{i_2},\,\cdots,\,v_{i_n}\rangle &\supsetneq \langle v_{i_2},\cdots,v_{i_n}\rangle \supsetneq \cdots \supsetneq \langle v_{i_n}\rangle\supsetneq 0,\\
			\alpha^p_{i_1} &<\alpha^p_{i_2} <\cdots \cdots\cdots\cdots<\alpha^p_{i_n}.
		\end{align*}
		By repeating this for all $1\,\leq\, i\,\leq\, m$, a parabolic bundle $F_*$ on $Y_L$ is constructed.
		Note that $F_*$ must be parabolic stable, because if $F'\,\subset\, F$ is any sub-bundle
		such that $\text{Par}\mu(F'_*)\,\geq \,\text{Par}\mu(F_*)$, then the equalities
		\[\text{par}deg(\gamma_*(F'_*))\,=\,\text{par}deg(F'_*)\ \ \text{ and }\ \ 
		\text{rank}(\gamma_*(F')) \,=\, m\cdot \text{rank}(F')\]
		would imply that $\gamma_*(F')\,\subset\, E$ violates the condition of parabolic stability
		for $E_*$. Thus $F_*\,\in\, \mathcal{N}^t_L$ for some $t\in P(\alpha).$ Moreover, the
		parabolic push-forward of $F_*$ under $\gamma$ as in Remark \ref{parabolic pushforward}
		coincides with $E_*$.
		
		Finally, if the number of parabolic points on $X$ is more than 1, then an exactly similar 
		argument with the obvious modifications will give the result.
	\end{proof}
	
	We note that if $|D|\,=\,s$, then each $\mathcal{N}^t_L$ is of dimension 
	\begin{align*}
		\dim(\mathcal{N}^t_L) &= \frac{r^2}{m^2}(\text{genus}(Y_L)-1)+1-g+sm\cdot
		\dfrac{\frac{r}{m}(\frac{r}{m}-1)}{2}\\
		&= (g-1)(\frac{r^2}{m}-1)+\dfrac{sr(\frac{r}{m}-1)}{2}\,\,\,\,\,[\because \text{genus}(Y_L)=m(g-1)+1].
	\end{align*}
	
	Since $\dim(\Malpha)\,= \,(g-1)(r^2-1)+s\cdot\dfrac{r(r-1)}{2}$, from
	Lemma \ref{parabolic fixed point} it follows that
	$$
	\dim(\Malpha)-\dim(\MLalpha)
	\,\geq\, \dim(\Malpha)-\dim(\mathcal{N}_L)\,=\, r^2(\dfrac{m-1}{m})(g-1+\frac{s}{2})
	\,\geq\, 3.
	$$
	
	\begin{corollary}\label{cor1}
		The codimension of the closed subscheme
		$$Z_\alpha\,=\, \underset{L\in\Gamma\setminus\{\mathcal{O}_X\}}{\bigcup}\MLalpha\,\subset\, \Malpha
		$$
		is at least three.
	\end{corollary}
	
	\section{The Brauer group of the stack $\mathcal{PN}_{\alpha,i}$}
	
	We continue with the notation of Section \ref{se2}. For a parabolic bundle $E_*$, let 
	$\mathbb{P}(E)_*$ denote the corresponding parabolic $\text{PGL}(r,k)$-bundle. If $E_*$ is 
	stable parabolic, then $\mathbb{P}(E)_*$ is a stable parabolic $\text{PGL}(r,k)$-bundle.
	
	\begin{lemma}\label{crucial}
		Let $U_\alpha\,:=\,\Malpha\setminus Z_\alpha$ (see Corollary \ref{cor1}). Take any
		$E_*\,\in\, U_\alpha$, and let $E'_*\,:=\, E_*\otimes L$ for some $L\,\in \,\Gamma$. Then
		there exists a unique isomorphism of parabolic ${\rm PGL}(r,k)$-bundles between
		$\mathbb{P}(E)_*$ and $\mathbb{P}(E')_*$.
	\end{lemma}
	
	\begin{proof}
		Since the parabolic structure of $E'_*$ is induced from $E_*$, the canonical isomorphism 
		between the underlying $\text{PGL}(r,k)$-bundles $\mathbb{P}(E)$ and $\mathbb{P}(E')$ 
		actually gives an isomorphism of parabolic $\text{PGL}(r,k)$-bundles.  If there are two 
		isomorphisms between $\mathbb{P}(E)_*$ and $\mathbb{P}(E')_*$, then we get an automorphism 
		of $\mathbb{P}(E)_*$.
		
		Let $f_*$ be a nontrivial automorphism of $\mathbb{P}(E)_*$. 
		It evidently induces an automorphism $f$ of the underlying projective bundle $\mathbb{P}(E)$. 
		Then there exists a line bundle $L_0$ and an isomorphism of vector bundles 
		\[\widetilde{f}\,:\,E\,\longrightarrow\, E\otimes L_0\] which induces $f$. Taking determinant
		it follows that $L_0\,\in\, \Gamma$. Moreover, it is clear that $\widetilde{f}$ respects
		the parabolic structures on $E$ and $E\otimes L_0$, so we get an isomorphism of parabolic bundles
		\[\widetilde{f}_*\,:\, E_*\,\longrightarrow\, E_*\otimes L_0.\]
		But $E_*\,\in\, U_{\alpha}$, thus $L_0$ must be trivial. Hence $\widetilde{f}_*$ is an
		automorphism of $E_*$. Since $E_*$ is parabolic stable, it must be a scalar multiple of
		the identity map. But then the induced map $f_*$ must be the identity morphism.
	\end{proof} 
	
	The following proposition should be well-known to the experts, but we are providing a proof as we could not find it in the literature, and also for completeness.
	
	\begin{proposition}\label{brauer isomrphism of stacks}
	Let $\mathcal{X}$ be a locally Noetherian smooth Deligne-Mumford stack, and let
$\mathcal{Z}\,\hookrightarrow\, \mathcal{X}$ be a closed substack of codimension at least two. Then
$$H^2_{\text{\'et}}(\mathcal{X}\setminus\mathcal{Z},\,\mathbb{G}_m)\,\simeq\,
H^2_{\text{\'et}}(\mathcal{X},\,\mathbb{G}_m).$$
	\end{proposition}
	
\begin{proof}
Consider an \'etale cover $V\,\longrightarrow\, \mathcal{X}$. The pull-back of $\mathcal{Z}$ produces a 
closed subscheme $Z'$ of $V$. We still have $\codim_V(Z')\,\geq\, 2$ \cite[Lemma 5.1]{GS15}.

Let $V^{\times i}\,= \,V\times V\times\cdots \times V$ \,($i+1$ times). Now, consider the descent spectral sequence
$$E_1^{i,j}\,=\,H^j_{\text{\'et}}(V^{\times i},\, \mathbb{G}_m)\,\implies\, H^{i+j}_{\text{\'et}}(\mathcal{X},
\,\mathbb{G}_m), \,\,\,(i)$$
		and similarly
$$E_1^{i,j}\,=\,H^j_{\text{\'et}}((V\setminus Z')^{\times i},\, \mathbb{G}_m)\,\implies\,
H^{i+j}_{\text{\'et}}(\mathcal{X}\setminus\mathcal{Z},\,\mathbb{G}_m).\,\,\,(ii)$$
By purity statements for \'etale cohomology for smooth schemes \cite[Theorem 6.1]{Ce}, we know that
the open embedding $V\setminus Z'\,\hookrightarrow\, V$ induces isomorphisms on the group of units, the
Picard group and the cohomological Brauer group. Thus the $E_1^{i,j}$-terms in $(i)$ and $(ii)$ are
isomorphic for all $j\leq 2$ and all $i$. This in turn induces isomorphisms on the convergence
terms for all $i+j\leq 2$, and thus we have $H^2_{\text{\'et}}(\mathcal{X}\setminus\mathcal{Z},\,
\mathbb{G}_m)\,\simeq\, H^2_{\text{\'et}}(\mathcal{X},\,\mathbb{G}_m).$
	\end{proof}

	Next, consider 
	\[\xymatrix{
		& \mathcal{PN}_{\alpha,i} \ar[d]^{p}\\
		\Malpha \ar[r]^{\phi_\alpha} & \PNalpha
	}\]
	where $\phi_\alpha\,:\, \Malpha\,\longrightarrow \,\PNalpha$ denotes the quotient by $\Gamma$, and
$p\,:\,\mathcal{PN}_{\alpha,i}\,\longrightarrow\, \PNalpha$ is the coarse moduli space map. $\mathcal{PN}_{\alpha,i}$ is a smooth Deligne-Mumford stack, as the automorphism groups of its points are finite.\\
	Note that $\mathcal{PN}_{\alpha,i}$ is the quotient stack $[\Malpha/\Gamma]$, while $\PNalpha
\,=\,\Malpha/\Gamma$ is the GIT quotient, and thus the map $p$ is natural map
	$$p\,:\,[\Malpha/\Gamma]\,\longrightarrow\, \Malpha/\Gamma.$$
	Thus, by \cite[\href{https://stacks.math.columbia.edu/tag/075T}{Tag 075T},\,Lemma 100.33.7]{stacks-project}, $p$ is smooth. Since $\phi_\alpha$ is a finite morphism, clearly $\codim_{\PNalpha}(\phi_\alpha(Z_\alpha))\geq 2$ by Corollary \ref{cor1}. Since smooth morphisms of Artin stacks are codimension-preserving \cite[Lemma 5.1]{GS15}, 
	we conclude that the closed substack
	$$\mathcal{Z}_\alpha\,:=\,p^{-1}(\phi_\alpha(Z_\alpha))\,\hookrightarrow\,\mathcal{PN}_{\alpha,i}$$
	also has codimension two. \\

	\begin{corollary}\label{isomorphism of brauer group for stack}
		Let $PN^{sm}_{\alpha,i}$ denote the smooth locus of\, $\PNalpha$.
Then $\Br(\mathcal{PN}_{\alpha,i})\,\simeq\, \Br(PN^{sm}_{\alpha,i}).$	  
	\end{corollary}

	\begin{proof}
		Let $\mathcal{Z}_{\alpha} \,=\, p^{-1}Z_{\alpha}$, with $Z_\alpha$ as in Corollary \ref{cor1}.
Use Proposition \ref{brauer isomrphism of stacks} for $\mathcal{Z}_\alpha\,\hookrightarrow\, \mathcal{PN}_{\alpha,i}$, together with Lemma \ref{crucial}, which says that the map $p$ restricts to an isomorphism over $\phi_\alpha(U_\alpha)=\PNalpha\setminus\phi_\alpha(Z_\alpha)$.
	\end{proof}

	This also completes the proof of Theorem \ref{thi} $(1)$.

	\section{Description of the Brauer groups for sufficiently small weights}

	From this section onward, we allow ourselves to work with parabolic structures which can be consisting of partial flags as well. We shall also assume from now on the following condition on the parabolic structure :
	\begin{align}\label{coprime condition on parabolic structure} 
		g.c.d.(r,d,m^1_{1},\cdots,m^1_{l_1},\cdots,m^i_j,\cdots,m^s_1,\cdots,m^s_{l_s})=1.
	\end{align}
	(See Definition \ref{pardef}).  Assume that $\alpha$ be a sufficiently small generic weight satisfying these conditions. It is known that such a generic weight exists \cite[Proposition 3.2]{BY99}. Since $\alpha$ is chosen sufficiently small, there is a morphism
	\begin{equation}\label{ef}
		\pi\,:\,\Malpha\,\longrightarrow\, \textnormal{M}^{ss}_\xi
	\end{equation}
	that sends a parabolic bundle to the underlying vector bundle
	by simply forgetting the parabolic structure. $\Gamma$ acts on both $\Malpha$ and $M^{ss}_\xi$, and
$\pi$ in \eqref{ef} is equivariant under this action. Restrict $\pi$ over the stable locus
$\Mxi\subset M^{ss}_\xi$. Each fiber of $\pi\,:\, \pi^{-1}(\Mxi)\,\longrightarrow\, \Mxi$ is a product of
flag varieties. Note that $\pi$ is flat (this is sometimes called the ``miracle flatness theorem''
\cite[\href{https://stacks.math.columbia.edu/tag/00R4}{Tag 00R4}, Lemma 10.128.1]{stacks-project}).  In fact,
$\pi$ is an \'etale locally trivial fibration \cite[Theorem 1.3]{PP21}. The closed subset
\begin{equation}\label{efp2}
\mathcal{Z}\,:=\, \underset{L\in\Gamma\setminus \{\mathcal{O}_X\}}{\bigcup}M^L_\xi\, \subset\, \Mxi
\end{equation}
(see \eqref{efp}) has codimension at least three \cite[Corollary 2.2]{BH10}. Define
\begin{equation}\label{efp3}
\mathcal{U}\,:=\, \Mxi\setminus \mathcal{Z}.
\end{equation}
	
Consider $\pi^{-1}(\mathcal{U})\,\subset\,\pi^{-1}(\Mxi)$. Clearly the action of $\Gamma$ on $\Mxi$ restricts 
to a free action on $\mathcal{U}$, and hence the action of $\Gamma$ on $\pi^{-1}(\mathcal{U})$ is also free; 
indeed, if $E_*$ is an $L$-fixed point, then the underlying bundle $E$ is also fixed by $L$. Thus, if 
$PN^{sm}_{\alpha,i}$ (respectively, $N(r)^{sm}_i$) is the smooth locus of $\PNalpha$ (respectively, 
$N(r)_i$), then we have $$\mathcal{U}/\Gamma\subset N(r)^{sm}_i \,\,\,\ \text{ and } \,\,\,  
\pi^{-1}(\mathcal{U})/\Gamma\subset PN^{sm}_{\alpha,i}.$$ A straightforward codimension estimate shows that
\begin{align}\label{codimension estimates}
\codim_{\Malpha}(\Malpha\setminus\pi^{-1}(\mathcal{U}))	\,= 
\,\codim_{\Mxi}(\Mxi\setminus\mathcal{U})\,\geq\, 3.
\end{align}
	
	In view of these codimension estimates, we have
	\begin{align}\label{brauer groups of open sets}
		\Br(\mathcal{U}/\Gamma)=\Br(N(r)^{sm}_i)\,\,\text{and}\,\,\Br(\pi^{-1}(\mathcal{U})/\Gamma)=\Br(PN_{\alpha,i}^{sm}).
	\end{align}
	
	This enables us to work with finite \'etale morphisms in what follows. Consider the diagram
	\begin{align}\label{diagram 1}
		\xymatrix{
			\pi^{-1}(\mathcal{U}) \ar[r]^{\phi'} \ar[d]_{\pi} & \pi^{-1}(\mathcal{U})/\Gamma  \ar[d]^{\pi'} \\
			\mathcal{U} \ar[r]^{\phi} & \mathcal{U}/\Gamma}
	\end{align}
	where $\phi$ and $\phi'$ are both finite \'etale (due to the removal of all the fixed
	points), and $\pi,\, \pi'$ are both fibrations in the \'etale topology. The typical fibers of
	$\pi$ and $\pi'$ are the same, with both being isomorphic to the product of flag
varieties. In fact the diagram in \eqref{diagram 1} is Cartesian. Denote the common typical fiber of the two fiber bundles by $F$.
	
	\begin{proposition}\label{brauer group iso 1}
Let $\alpha$ be a sufficiently small weight, and $PN^{sm}_{\alpha,i}$ (respectively, $N(r)^{sm}_i$) be the smooth locus of $\PNalpha$ (respectively, $N(r)_i$), where $i\equiv d (mod \,r)$. Let $\mathcal{U}\subset \Mxi$ be as above. Then $\Br(\,\text{PN}^{sm}_{\alpha,i}\,)$ can be identified with the kernel of the natural surjective map
$$\Br(\,\textnormal{N}(r)^{sm}_i\,)\,=\, \Br(\mathcal{U}/\Gamma)\,\longrightarrow\, \Br(\mathcal{U})
\,=\,\Br(\Mxi)\,\cong\, \mathbb{Z}/\delta\mathbb{Z},$$
		where $\delta\,=\,{\rm g.c.d.}(r,\,d).$
	\end{proposition}
	
	\begin{proof}
		The Leray spectral sequence for the map $\pi\,:\,\pi^{-1}(\mathcal{U})\,\longrightarrow\, \mathcal{U}$
		gives us the exact sequence
		\begin{align}\label{exact seq 1}
			0\,\longrightarrow\, H^1(\mathcal{U},\,\pi_*\mathbb{G}_m)\,
\xrightarrow{\,\,\,\pi^*\,\,}\,\Pic(\pi^{-1}(\mathcal{U}))\,\longrightarrow\, H^0(\mathcal{U},\,R^1\pi_*\mathbb{G}_m)
		\end{align}
		We have $\pi_*\mathcal{O}_{\pi^{-1}(\mathcal{U})}\,=\,
\mathcal{O}_{\mathcal{U}}$ since $\pi$ is flat and global sections of the fibers are trivial; this implies $\pi_*\mathbb{G}_m=\mathbb{G}_m.$
		Also, since $\mathcal{U}$ is simply connected, $R^1\pi_*\mathbb{G}_m$ is the constant sheaf with stalk $\Pic(F)$, where $F$ is a typical fiber of $\pi$ (for details, see \cite[Lemma 3.1]{BD11}.)\\
		Similarly, for $\pi'\,:\,\pi^{-1}(\mathcal{U})/\Gamma\,\longrightarrow\, \mathcal{U}/\Gamma$ we get:
		\begin{align}\label{exact seq 2}
0\,\longrightarrow\, H^1(\mathcal{U}/\Gamma,\,\pi'_*\mathbb{G}_m)\,\xrightarrow{\,\,\pi'^*\,} \,
\Pic(\pi^{-1}(\mathcal{U})/\Gamma)\,\longrightarrow\, H^0(\mathcal{U}/\Gamma,\,R^1\pi'_*\mathbb{G}_m)
		\end{align}
		Again for similar reasons as above, $\pi'_*\mathbb{G}_m\,=\,\mathbb{G}_m$. We claim
that $R^1\pi'_*\mathbb{G}_m$ is the constant sheaf with stalk $\Pic(F)$. This is because, as remarked
above, $R^1\pi'_*\mathbb{G}_m$ becomes the constant sheaf
		\,\,$\underline{\Pic(F)}_{\mathcal{U}}$\, after being pulled back
to the \'etale cover $\mathcal{U}\,\xrightarrow{\,\,\phi\,}\,\mathcal{U}/\Gamma$, and the action of
$\Gamma$ on $\Pic(F)$ is trivial.\\
		
		Thus the exact sequences of \eqref{exact seq 1} and \eqref{exact seq 2}	become the following, respectively:
		\begin{align}
			& 0\,\longrightarrow\, Pic(\mathcal{U})\,\xrightarrow{\,\,\pi^*\,}\,
\Pic(\pi^{-1}(\mathcal{U}))\,\longrightarrow\, \Pic(F)\\
			& 	0\,\longrightarrow\, \Pic(\mathcal{U}/\Gamma)\,\xrightarrow{\,\,\pi'^*\,\,}\,
\Pic(\pi^{-1}(\mathcal{U})/\Gamma)\,\longrightarrow\, \Pic(F).
		\end{align}
		
		Consequently, we get the following diagram:
		\begin{align}\label{diagram 2}
			\xymatrix{
				0 \ar[r] & \Pic(\mathcal{U}/\Gamma) \ar[r]^(.45){\pi'^*} \ar[d]^{\phi^*} & \Pic(\pi^{-1}(\mathcal{U})/\Gamma) \ar[r]^(.6){j'^*} \ar[d]^{\phi'^*} & \Pic(F) \ar[d]^{g'}  \\
				0\ar[r] & \Pic(\mathcal{U}) \ar[r]^(.45){\pi^*} & \Pic(\pi^{-1}(\mathcal{U})) \ar[r]^(.6){j^*} & \Pic(F) 
			}
		\end{align}
		where $j,\,j'$ denote the isomorphism of $F$ with some fiber, and $g'$ is the induced map.
		
		\begin{claim}\label{induced map on pic(F) is identity}
			The homomorphism $g'$ in \eqref{diagram 2} is the identity map.
		\end{claim}
		
		\begin{proof}[{Proof of claim}]
			First, note that for a point $y\,\in\, \mathcal{U}/\Gamma$, identifying $F$ with $\pi'^{-1}(y)$,
			\begin{align*}
				\phi'^{-1}(F)= \phi'^{-1}(\pi'^{-1}(y))  = \pi^{-1}(\phi^{-1}(y))
				& \simeq \pi^{-1}(\{y\}\times \Gamma)\\
				& = \pi^{-1}(y) \times \Gamma \\
				& \simeq F \times \Gamma,
			\end{align*}
			and hence $\Pic(\phi'^{-1}(F))\,= \,\bigoplus_{i=1}^{|\Gamma|}\Pic(F)$. Therefore,
			$g'$ can be expressed as the composition
			\[\Pic(F)\,=\, \Pic(\pi'^{-1}(y))\,\xrightarrow{\,\phi'^*\,}\,\bigoplus_{i=1}^{|\Gamma|}\Pic(F)
			\,\longrightarrow\,\Pic(F),\]
			where the last arrow is the projection to any one of the components. It is easy to see
			that this is the identity map, which proves the claim.
		\end{proof}		

We know that ${\rm Pic}(\Mxi) \,=\, {\mathbb Z}$ \cite[p.~55, Th\'eor\`eme B]{DN}. Since
the codimension of $\mathcal{Z}\,\, \subset\, \Mxi$ (see \eqref{efp2}) is at least three (noted following
\eqref{efp2}), we now conclude that ${\rm Pic}(\mathcal{U}) \,=\, {\mathbb Z}$. The Picard group
of $\Malpha$ is isomorphic to ${\mathbb Z}^{\oplus N}$, where
$$
N\,=\, 1+\sum_{i=1}^s (l_i-1)\,=\, 1-s +\sum_{i=1}^s l_i
$$
\cite{MS80}, \cite{LS}. Now from \eqref{codimension estimates} it follows that
$\text{Pic}(\pi^{-1}(\mathcal{U}))\,=\, {\mathbb Z}^N$.

Note that $\Gamma$ acts trivially on both $\Pic(\mathcal{U})$ and $\Pic(\pi^{-1}(\mathcal{U}))$ 
and both are torsion-free (as shown above), which implies that 
\[H^1(\Gamma,\, \Pic(\mathcal{U}))
\,=\, {\rm Hom}(\Gamma,\,\Pic(\mathcal{U}))\,=\, 0\,\,\,\, \text{ and}
\]
\[
H^1(\Gamma,\,
\Pic(\pi^{-1}(\mathcal{U})))\,=\, {\rm Hom}(\Gamma,\,\Pic(\pi^{-1}(\mathcal{U})))\,=\,0,\]
		since $\Gamma$ is finite. The five-term exact sequence corresponding to the Hochschild-Serre spectral sequence
		associated to $\phi$ and $\phi'$ \cite[III Theorem 2.20]{Mi} then produces following two exact sequences:
		\begin{flalign}\label{ses 1}			
			&\scalebox{0.9}{\xymatrix@R=2pc @C=1.5pc{
					0 \ar[r] & \chi(\Gamma) \ar[r]^(.4){f} & \Pic(\mathcal{U}/\Gamma) \ar[r]^{\phi^*} & \Pic(\mathcal{U}) \ar[r] & H^2(\Gamma, k^*) \ar[r] & \Br(\mathcal{U}/\Gamma) \ar[r] & \Br(\mathcal{U})\simeq \mathbb{Z}/\delta\mathbb{Z} \ar[r] & 0
			}}\\
			&\scalebox{0.85}{\xymatrix@R=2pc @C=1.5pc{
					0 \ar[r] & \chi(\Gamma) \ar[r]^(.3){f'} & \Pic(\pi^{-1}(\mathcal{U})/\Gamma) \ar[r]^{\phi'^*}  & \Pic(\pi^{-1}(\mathcal{U})) \ar[r] & H^2(\Gamma, k^*) \ar[r] &  \Br(\pi^{-1}(\mathcal{U})/\Gamma) \ar[r] & \Br(\pi^{-1}(\mathcal{U})) \ar@{=}[r] & 0
				}\label{ses 2}
			} 
		\end{flalign}
		whose rows are exact. Note that due to the codimension estimates in \eqref{codimension estimates},
		$$\Br(\mathcal{U})\,\cong\, \Br(\Mxi)\,\cong\, \mathbb{Z}/\delta\mathbb{Z}\,\,\, \ \text{ and }\, \,
\, \ \Br(\pi^{-1}(\mathcal{U}))\,\cong\, \Br(\Malpha)\,=\,0
$$
\cite[Theorem 1.1]{BD11}. Let $M\,:=\,
\text{coker}(f)$ and $M'\,:=\,\text{coker}(f')$. Using the snake lemma in the diagram
		\begin{align*}
			\xymatrix{
				0 \ar[r] & \chi(\Gamma) \ar[r]^(.4){f} \ar[d]^{\cong} & \Pic(\mathcal{U}/\Gamma) \ar[r]^(.6){\phi^*} \ar@{^{(}->}[d]^{\pi'^*}  & M \ar[r] \ar@{^{(}->}[d] & 0 \\
				0 \ar[r] & \chi(\Gamma) \ar[r]^(.3){f'} & \Pic(\pi^{-1}(\mathcal{U})/\Gamma) \ar[r]^(.7){\phi'^*} & M' \ar[r] & 0
			}
		\end{align*}
		we conclude that
		\begin{align}\label{equation 1}
			\Pic(F) \,\simeq\, \dfrac{M'}{M}.
		\end{align}
		
		Set $H\,:=\, \text{coker}(\phi^*)$ and
		$H'\,:=\, \text{coker}(\phi'^*)$ (see the diagram in \eqref{diagram 1}).
		The group $H$ is cyclic of order $\delta$ \cite[Theorem 1.2]{BH10}.
		Using the snake lemma in the diagram
		\begin{align*}
			\xymatrix{
				0\ar[r] & M \ar[r] \ar@{^{(}->}[d] & \Pic(\mathcal{U}) \ar[r] \ar@{^{(}->}[d] & H \ar[r] \ar@{^{(}->}[d] & 0\\
				0\ar[r] & M' \ar[r] & \Pic(\pi^{-1}(\mathcal{U})) \ar[r] & H' \ar[r] & 0
			}
		\end{align*} 
		and from (\ref{diagram 2}), it follows that
		\begin{align}\label{equation 2}
			0\,\longrightarrow \,\dfrac{M'}{M} \,\longrightarrow\, \Pic(F) \,\longrightarrow
			\,\dfrac{H'}{H}\,\longrightarrow\, 0.
		\end{align}
		Composing the first map with (\ref{equation 1}) gives
		\[\Pic(F)\,\simeq\,\dfrac{M'}{M}\,\longrightarrow\, \Pic(F)\]
		which is the identity map in Claim \ref{induced map on pic(F) is identity}, and hence 
		\[\dfrac{H'}{H}=0\implies H\simeq H'.\]
		
		{}From the exact sequences in \eqref{ses 1} and \eqref{ses 2} we get the following two exact sequences:
		\begin{enumerate}[(i)]
			\item $0\,\longrightarrow \,H \,\longrightarrow \,H^2(\Gamma,\, k^*)
			\,\longrightarrow\, \Br(\mathcal{U}/\Gamma)\,\longrightarrow\,\Br(\mathcal{U})\simeq\mathbb{Z}/\delta\mathbb{Z}\,\longrightarrow\,
			0$, and
			
			\item $0\,\longrightarrow\, H'\,\longrightarrow\, H^2(\Gamma,\,k^*)
			\,\longrightarrow\, \Br(\pi^{-1}(\mathcal{U})/\Gamma)\,\longrightarrow\, 0.$
		\end{enumerate}
		Since both $\PNalpha$ and $ N(r)_i $ are normal projective varieties and the open subsets $\mathcal{U}/\Gamma$ and $\pi^{-1}(\mathcal{U})/\Gamma$ both have complement of codimension 2, it follows that $\pi^*$ induces isomorphism $$k^*\,=\,H^0(\mathcal{U}/\Gamma,\,\mathbb{G}_m)\,\longrightarrow\,
H^0(\pi^{-1}(\mathcal{U})/\Gamma,\,\mathbb{G}_m)\,=\,k^*,$$
		and thus it induces isomorphism $H^2(\Gamma,\,k^*)\,\longrightarrow \,H^2(\Gamma,\,k^*)$, which
takes $H$ to $H'$.
		Thus, from the exact sequences (i) and (ii) we finally conclude that
		\begin{align*}
			\Br(\mathcal{U}/\Gamma)\,\simeq\,\dfrac{H^2(\Gamma, k^*)}{H'}\simeq \dfrac{H^2(\Gamma, k^*)}{H}\,\simeq\,ker\left(\Br(\pi^{-1}(\mathcal{U})/\Gamma)\rightarrow \mathbb{Z}/\delta\mathbb{Z}\right).
		\end{align*}
		Applying \eqref{brauer groups of open sets} proves our claim. 
	\end{proof}
	
		\section{Brauer groups for arbitrary generic weights}

		Let us consider the case of an arbitrary generic weight corresponding to a parabolic structure satisfying the condition (\ref{coprime condition on parabolic structure}), namely
		$$g.c.d.(r,d,m^1_{1},\cdots,m^1_{l_1},\cdots,m^i_j,\cdots,m^s_1,\cdots,m^s_{l_s})=1.$$
		For any such generic weight $\alpha$, let us denote as usual the smooth locus of $\PNalpha$ by $PN^{sm}_{\alpha,i}$. Let $\alpha$ and $\beta$ be two generic weights separated by a single wall $W$ (cf. 
		\cite[\S~2]{BY99}). In \cite[\S~5, page 11]{Th02} it is shown that there exists a closed 
		subscheme $Y_\alpha\,\hookrightarrow\, \Malpha$ whose image is precisely the locus of those 
		parabolic bundles that are stable for weight $\alpha$ but not stable for weight $\beta$, and thus 
		$\Malpha\setminus Y_\alpha$ is the open subset of $\Malpha$ consisting precisely of those 
		parabolic bundles which are both $\alpha$-stable as well as $\beta$-stable. Similarly, 
		there exists a closed subscheme $Y_\beta\,\hookrightarrow\, \Mbeta$ with exactly similar 
		properties as $Y_\alpha$ when $\alpha$ and $\beta$ are interchanged.
		Thus there exists a natural isomorphism
		\begin{align}\label{isomorphic open sets in Malpha and Mbeta}
			\Malpha\setminus Y_\alpha \,\simeq\, \Mbeta\setminus Y_\beta.
		\end{align}
		Moreover, both $Y_\alpha$ and $Y_\beta$ have codimension at least 2 in $\Malpha$ and 
		$\Mbeta$ respectively.
		
		From their descriptions it is clear that both $U_\alpha\,=\,\Malpha\setminus Y_\alpha$ and 
		$U_\beta\,=\, \Mbeta\setminus Y_\beta$ are $\Gamma$-invariant, and their identification in 
		(\ref{isomorphic open sets in Malpha and Mbeta}) is $\Gamma$-equivariant. Thus, if we consider the closed subschemes $Z_\alpha$ and $Z_\beta$ of $\Malpha$ and $\Mbeta$ respectively as in Corollary \ref{cor1}, then $V_\alpha:= U_\alpha\setminus Z_\alpha$ and $V_\beta:= U_\beta\setminus Z_\beta$ are $\Gamma$-invariant open subsets of $\Malpha$ and $\Mbeta$ respectively; but now $\Gamma$ acts freely on these sets. It follows that the isomorphism in (\ref{isomorphic open sets in Malpha and Mbeta}) descends to an isomorphism
		\begin{align}
			V_\alpha/\Gamma \,\simeq\, V_\beta/\Gamma,
		\end{align}
		and moreover, under the quotients by the $\Gamma$-action $\phi_\alpha:\Malpha\,\longrightarrow\, \PNalpha$ and $\phi_\beta:\Mbeta\,
		\longrightarrow\,\textnormal{PN}_{\beta,i}$\,\,, we have $V_\alpha/\Gamma \subset PN^{sm}_{\alpha,i}$ and $V_\beta/\Gamma\subset PN^{sm}_{\beta,i}$. Of course, since $\phi_\alpha$ is a finite morphism, the complement of $V_\alpha/\Gamma$ in $PN^{sm}_{\alpha,i}$ will still be of codimension at least 2; the same is true if we replace $\alpha$ by $\beta$. Hence we conclude that when $\alpha$ and $\beta$ are in adjacent chambers separated by a single wall,
		\begin{align}\label{brauer group isomorphism for generic weights}
			\Br(PN_{\alpha,i}^{sm})\,\cong\, \Br(PN^{sm}_{\beta,i}).
		\end{align}
		
		Now, since there are only finitely many walls, we can arrange the collection of chambers in 
		a sequence, say $C_1,\,C_2,\,\cdots,\,C_N$, such that $C_1$ contains a sufficiently small weight, 
		and for each $1\,\leq \,j\,<\, N$, the chambers $C_j$ and $C_{j+1}$ are separated by a single wall. Choose 
		generic weights $\alpha_j\,\in\, C_j$ for each $j$, such that $\alpha_1$ is a sufficiently small weight. By (\ref{brauer group isomorphism for generic weights}), we have
		\[\Br(\textnormal{PN}^{sm}_{\alpha_j,i})\,\cong\, \Br(\textnormal{PN}^{sm}_{\alpha_{j+1},i})\ \ \,\,
		\forall\ \ 1\,\leq\, j\,\leq\,N.\]
		
		Thus we obtain the following:
		
		\begin{corollary}
			For any two generic weights $\alpha$ and $\beta$ corresponding to a parabolic structure satisfying the condition $g.c.d.(r,d,m^1_{1},\cdots,m^1_{l_1},\cdots,m^i_j,\cdots,m^s_1,\cdots,m^s_{l_s})=1,$ we have
			$$\Br(\,PN^{sm}_{\alpha,i}\,)\,\cong\,\Br(\,\textnormal{PN}^{sm}_{\beta,i}\,).$$
			Thus, Proposition \ref{brauer group iso 1} remains valid for arbitrary generic weights satisfying this condition.
		\end{corollary}
		This also completes the proof of Theorem \ref{thi} $(2)$.
		\section{Brauer groups for certain non-generic weights}\label{non-generic weight section}
		It is known \cite[\S~2]{BY99} that the set of all possible  weights corresponding to a given quasi-parabolic structure forms a simplex, and there are finitely many hyperplanes, called \textit{walls}, which intersect this simplex. The generic weights are precisely those weights which do not lie on these walls. We call a weight to be non-generic if it lies on a wall. \\
		\begin{proposition}\label{non-generic weight proposition}
			Let us consider a parabolic structure satisfying $g.c.d.(r,d,m^1_{1},\cdots,m^1_{l_1},\cdots,m^i_j,\cdots,m^s_1,\cdots,m^s_{l_s})=1$. Let $\gamma$ be a non-generic weight corresponding to this parabolic structure, lying on a single wall (that is, it does not lie on the intersection of two or more walls). If $\alpha$ is a generic weight (we know a generic weight exists \cite[Proposition 3.2]{BY99}), we have
			\[\Br(PN^{sm}_{\alpha,i})\simeq \Br(PN^{sm}_{\gamma,i}).\]
		\end{proposition}
		
		\begin{proof}
			Suppose that the weight $\gamma$ lies on a single wall $W$. It follows that in a small enough neighborhood of $\gamma$, we can choose two generic weights $\alpha$ and $\beta$ which are separated by $W$ only, and $\gamma$ is the point of intersection between $W$ and the line joining $\alpha$ and $\beta$.  By \cite[Lemma 3.6]{BH95}, any $\alpha$-parabolic stable bundle is $\gamma$-parabolic semistable, which induces a $\Gamma$-equivariant morphism $\varphi_\alpha\,:\, M_\alpha\,\longrightarrow\, M_\gamma$ by simply replacing the weights. Note that $M_\gamma$ is no longer smooth. Let $U_\gamma\subset M_\gamma$ be the open subset consisting of $\gamma$-parabolic stable bundles. From \cite[Theorem 3.1]{BH95} it follows that $\varphi_\alpha$ is an isomorphism between the open subsets $U_\alpha:=\varphi_\alpha^{-1}(U_\gamma)$ and $U_\gamma$\,, and moreover, both of them have complements of codimension at least 2. Since $U_\gamma$ is $\Gamma$-invariant, this isomorphism is $\Gamma$-equivariant. The rest of the argument is similar to above. Namely, if we consider the closed subschemes $Z_\alpha$ and $Z_\gamma$ of $M_\alpha$ and $M_\gamma$ respectively as in Corollary \ref{cor1}, then $V_\alpha:=U_\alpha\setminus Z_\alpha$ and $V_\gamma:= U_\gamma\setminus Z_\gamma$ are also $\Gamma$-equivariantly isomorphic under $\varphi_\alpha$, and thus descends to an isomorphism
			\[V_\alpha/\Gamma \simeq V_\gamma/\Gamma.\]
			Since $\Gamma$ acts freely on both $V_\alpha$ and $V_\gamma$, we have $V_\alpha/\Gamma\subset PN^{sm}_{\alpha,i}$ and its complement will still have codimension 2. Of course, the same is true for $V_\gamma/\Gamma\subset PN^{sm}_{\gamma,i}$ as well.
			From this we conclude that 
			$$\Br(PN^{sm}_{\alpha,i})\simeq \Br(PN^{sm}_{\gamma,i})\,.$$
			Note that $\Br(PN^{sm}_{\alpha,i})$ is already known by Proposition \ref{brauer group iso 1}. 
		\end{proof}
		
\section*{Acknowledgements}

We thank the referee for helpful comments.

	\end{document}